\documentclass[12pt]{amsart}

\usepackage{framed}
\usepackage{braket}
\usepackage{amsmath,amssymb,amsthm}
\usepackage{color}
\usepackage{bm}
\usepackage[all]{xy}
\usepackage{amscd}
\usepackage{graphicx}
\usepackage{ascmac}
\usepackage{BOONDOX-frak, mathrsfs}
\usepackage[top=30truemm,bottom=30truemm,left=25truemm,right=25truemm]{geometry}
\usepackage{mathtools}
\mathtoolsset{showonlyrefs=true}
\usepackage{tikz}
\usetikzlibrary{cd}

\makeatletter
\@addtoreset{equation}{section}

\makeatother

\newcommand{\Z}{\mathbb{Z}}
\newcommand{\N}{\mathbb{N}}
\newcommand{\Q}{\mathbb{Q}}

\newcommand{\F}{\mathbb{F}}

\newcommand{\fm}{\mathfrak{m}}

\newcommand{\cC}{\mathcal{C}}

\newcommand{\cF}{\mathcal{F}}
\newcommand{\cG}{\mathcal{G}}

\newcommand{\cO}{\mathcal{O}}
\newcommand{\cP}{\mathcal{P}}

\newcommand{\cR}{\mathcal{R}}

\newcommand{\ol}[1]{\overline{#1}}

\DeclareMathOperator{\Gal}{Gal}

\DeclareMathOperator{\Ker}{Ker}

\DeclareMathOperator{\Fitt}{Fitt}

\DeclareMathOperator{\ab}{ab}

\DeclareMathOperator{\ord}{ord}
\DeclareMathOperator{\Hom}{Hom}
\DeclareMathOperator{\Ext}{Ext}

\DeclareMathOperator{\rank}{rank}

\usepackage[OT2,T1]{fontenc}
\DeclareSymbolFont{cyrletters}{OT2}{wncyr}{m}{n}
\DeclareMathSymbol{\Sha}{\mathalpha}{cyrletters}{"58}

\let\oldenumerate\enumerate
\renewcommand{\enumerate}{
   \oldenumerate
   \setlength{\itemsep}{1pt}
   \setlength{\parskip}{0pt}
   \setlength{\parsep}{0pt}
}
\let\olditemize\itemize
\renewcommand{\itemize}{
   \olditemize
   \setlength{\itemsep}{1pt}
   \setlength{\parskip}{0pt}
   \setlength{\parsep}{0pt}
}

\theoremstyle{plain}
\newtheorem{thm}{Theorem}[section]
\newtheorem{lem}[thm]{Lemma}

\newtheorem{prop}[thm]{Proposition}
\newtheorem{cor}[thm]{Corollary}

\newtheorem{ass}[thm]{Assumption}

\theoremstyle{definition}
\newtheorem{defn}[thm]{Definition}
\newtheorem{setting}[thm]{Setting}
\newtheorem{rem}[thm]{Remark}
\newtheorem{eg}[thm]{Example}

\usepackage{color}

\DeclareMathOperator{\gen}{gen}

\DeclareMathOperator{\ch}{char}
\DeclareMathOperator{\Frac}{Frac}
\DeclareMathOperator{\Tor}{Tor}
\DeclareMathOperator{\RG}{\mathsf{R}\Gamma}

\title
{Minimal resolutions of Iwasawa modules}
\author[T.~Kataoka]{Takenori Kataoka}
\address{Department of Mathematics, Faculty of Science Division II, Tokyo University of Science.
1-3 Kagurazaka, Shinjuku-ku, Tokyo 162-8601, Japan}
\email{tkataoka@rs.tus.ac.jp}
\author[M.~Kurihara]{Masato Kurihara}
\address{Faculty of Science and Technology, Keio University.
3-14-1 Hiyoshi, Kohoku-ku, Yokohama, Kanagawa 223-8522, Japan}
\email{kurihara@keio.jp}
\date{\today}

\begin{document}

\maketitle

\begin{abstract}
In this paper, we study the module-theoretic structure of classical Iwasawa modules.
More precisely, for a finite abelian $p$-extension $K/k$ of totally real fields and the cyclotomic $\Z_p$-extension $K_{\infty}/K$, we consider $X_{K_{\infty},S}=\Gal(M_{K_{\infty},S}/K_{\infty})$ where 
$S$ is a finite set of places of $k$ containing all ramifying places in $K_{\infty}$ and  archimedean places, and 
$M_{K_{\infty},S}$ is the maximal abelian pro-$p$-extension of $K_{\infty}$ unramified outside $S$.
We give lower and upper bounds of the minimal numbers of generators and of relations of $X_{K_{\infty},S}$ as a $\Z_p[[\Gal(K_{\infty}/k)]]$-module, using the $p$-rank of $\Gal(K/k)$.
This result explains the complexity of $X_{K_{\infty},S}$ as a $\Z_p[[\Gal(K_{\infty}/k)]]$-module when the $p$-rank of $\Gal(K/k)$ is large.  
Moreover, we prove an analogous theorem in the setting that $K/k$ is non-abelian. 
We also study the Iwasawa adjoint of $X_{K_{\infty},S}$, and the minus part of the unramified Iwasawa module for a CM-extension.  
In order to prove these theorems, we systematically study the minimal resolutions of $X_{K_{\infty},S}$.
\end{abstract}

\section{Introduction}\label{sec:intro}

Throughout this paper we fix a prime number $p$.
We write $F_{\infty}$ for the cyclotomic $\Z_p$-extension of $F$ for any number field $F$.

Let $K/k$ be a finite abelian $p$-extension of totally real fields
(see Theorem \ref{thm:bound2} for the non-abelian case).  
We consider the abelian extension $K_{\infty}/k$, whose Galois group we denote by $\cG=\Gal(K_{\infty}/k)$. 
Suppose that $S$ is a finite set of places of $k$, containing all archimedean places and 
all places that ramify in $K_{\infty}$. 
In particular, $S$ contains all $p$-adic places. 
Let $M_{K_{\infty},S}$ denote the maximal abelian pro-$p$-extension of $K_{\infty}$ unramified outside $S$. 
Our main purpose in this paper is to study the classical Iwasawa module 
$X_{K_{\infty},S}=\Gal(M_{K_{\infty},S}/K_{\infty})$ over the Iwasawa algebra $\cR = \Z_p[[\cG]]$.

Define $I_{\cG}$ 
to be the augmentation ideal of $\cR = \Z_p[[\cG]]$, namely 
$I_{\cG}=\Ker (\Z_p[[\cG]] \to \Z_p)$. 
We write $Q(\cR)$ for the total quotient ring of $\cR$. 
We consider an $\cR$-submodule $\cR^{\sim}$ of $Q(\cR)$, 
which consists of elements $x \in Q(\cR)$, satisfying $x I_{\cG} \subset \cR$.
This is the module of pseudo-measures of $\cG$ in the sense of Serre.
The $p$-adic $L$-function  of Deligne and Ribet is an element $g_{K_{\infty}/k,S}$ in $\cR^{\sim}$, 
satisfying the following property.  
Suppose that $\kappa:\cG \to \Z_p^\times$ is the cyclotomic character. 
For a character $\psi$ of $\cG$ of finite order with values in an algebraic 
closure $\overline{\Q}_{p}$ of $\Q_p$ and for a positive integer $n$, one can extend a character $\kappa^n \psi: \cG \to \overline{\Q}_{p}^{\times}$ to a ring homomorphism $\cR \to \overline{\Q}_{p}$, and also to $\cR^{\sim} \to \overline{\Q}_{p}$.
Then $g_{K_{\infty}/k,S}$ satisfies  
$$\kappa^n\psi(g_{F_{\infty}/k,S})=L_{S}(1-n, \psi\omega^{-n})$$
for any character $\psi$ of $\cG$ of finite order and for any positive integer $n \in \Z_{>0}$,
where $L_{S}(s, \psi\omega^{-n})$ is the $S$-truncated $L$-function, 
and $\omega$ is the Teichm\"{u}ller character. 

Put $G = \Gal(K_{\infty}/k_{\infty})$.
In \cite[Theorem 3.3]{GK15} and \cite[Theorem 4.1]{GK17}, as a refinement of the usual main conjecture, 
Greither and the second author computed the Fitting ideal of $X_{K_{\infty},S}$ as an $\cR$-module to obtain 
\begin{equation}\label{GKTFitt}
\Fitt_{\cR}(X_{K_{\infty},S}) = {\mathfrak a}_{G} I_{\cG} g_{K_{\infty}/k, S},
\end{equation}
where ${\mathfrak a}_{G}$ is a certain ideal of $\cR$ which is determined only by the group 
structure of $G$. 
The explicit description of ${\mathfrak a}_{G}$ is obtained in \cite[\S 1.2]{GKT20} by Greither, Tokio and the second author.
We do not explain this ideal ${\mathfrak a}_{G}$ in this paper, 
but only mention two facts. 
If $s$ is the $p$-rank of $G$ (i.e., $s = \dim_{\F_p}(\F_p \otimes_{\Z} G)$) and ${\mathfrak m}_{\cR}$ is the maximal ideal of $\cR$, 
then we have ${\mathfrak a}_{G} \subset {\mathfrak m}_{\cR}^{s(s-1)/2}$. 
Also, if $G$ is isomorphic to $(\Z/p^m)^{\oplus s}$, then  ${\mathfrak a}_{G}=(p^m\cR+I_{\cG})^{s(s-1)/2}$.
We also note here that the classical main conjecture in Iwasawa theory studies the character 
component $X_{K_{\infty},S}^{\psi}$ for $K$ which corresponds to the kernel of $\psi$. 
In this case, $G$ is cyclic, and only the case $s=1$ is studied.

The above computation of $\Fitt_{\cR}(X_{K_{\infty},S})$ suggests that $X_{K_{\infty},S}$ is complicated as an $\cR$-module when the $p$-rank $s$ of $G$ is large. 
To understand such complicatedness, we study in this paper the minimal numbers of generators and relations of $X_{K_{\infty},S}$. 
Let $\gen_{\cR}(X_{K_{\infty}, S})$ (resp.~$r_{\cR}(X_{K_{\infty}, S})$) be the 
minimal number of generators (resp.~of relations) of $X_{K_{\infty}, S}$ as an $\cR$-module. 

In order to state the main result of this paper, we need the maximal abelian pro-$p$-extension $M_{k, S}$ of $k$ unramified outside $S$.
By our choice of $S$, we have 
\[
k_{\infty} \subset K_{\infty} \subset M_{k, S}.
\]

Now we state the main result of this paper.
For any abelian group $A$, we define its $p$-rank by $\rank_p A = \dim_{\F_p}(\F_p \otimes_{\Z} A)$, which is finite in all cases we consider in this paper.

\begin{thm}\label{thm:bound1}
Let us write
\[
s = \rank_p \Gal(K_{\infty}/k_{\infty})= \rank_p G,
\quad
t = \rank_p \Gal(M_{k, S}/K_{\infty})
\]
for the $p$-ranks of the Galois groups.
Then we have
\[
\max \left\{ \frac{s(s+1)}{2}, t \right\} 
\leq \gen_{\cR}(X_{K_{\infty}, S})
\leq \frac{s(s+1)}{2} + t
\]
and
\[
r_{\cR}(X_{K_{\infty}, S})=\frac{s(s+1)(s+2)}{6}+ \gen_{\cR}(X_{K_{\infty}, S}).
\]
\end{thm}

\begin{rem}\label{rem:t=0}
We note that Leopoldt's conjecture for $k$ is equivalent to that the extension $M_{k, S}/k_{\infty}$ is finite though we do not assume it in Theorem \ref{thm:bound1}.
We get the equalities of the minimal numbers of generators in the following special cases.
\begin{itemize}
\item[(1)]
If Leopoldt's conjecture holds and $p \geq 3$, then we may take $K$ so that $K_{\infty} = M_{k, S}$.
In this case, we have $t = 0$, so the theorem says
\[
\gen_{\cR}(X_{M_{k, S}, S})
= \frac{s(s+1)}{2},
\]
where $s = \rank_p\Gal(M_{k, S}/k_{\infty})$.
\item[(2)]
In case $K_{\infty} = k_{\infty}$, we have $s = 0$, so the theorem says
\[
\gen_{\cR}(X_{k_{\infty}, S})
= t = \rank_p \Gal(M_{k, S}/k_{\infty}).
\]
Indeed, this follows directly from Lemma \ref{lem:Gab}.
\end{itemize}
Except for these cases, we have no theoretical method to determine the exact value of $\gen_{\cR}(X_{K_{\infty}, S})$ so far.
\end{rem}

\begin{rem}\label{rem:NumericalExamples}
In \S \ref{sec:examples} we give several numerical examples for $p=2$, $3$. 
We take $k=\Q$ and $K/\Q$ which is a real abelian extension such that $\Gal(K/\Q) \simeq (\Z/p\Z)^{\oplus s}$. 
Here, we pick up some typical examples from \S \ref{sec:examples}.
\begin{itemize}
\item[(1)] Take $p=3$. For two primes $\ell_i$ with $i=1$, $2$ such that $\ell_i \equiv 1$ (mod $3$), let $K$ be the unique $(\Z/3\Z)^{\oplus 2}$-extension over $\Q$ with conductor $\ell_1\ell_2$. 
Take $\ell_1=7$, and $\ell_2$ satisfying $\ell_2 \equiv 1$ (mod $9$), and 
consider $S=\{3, \ell_1, \ell_2, \infty\}$.
In this case $s=2$ and $t=1$. So Theorem \ref{thm:bound1} says that 
\[
3 \leq \gen_{\cR}(X_{K_{\infty}, S}) \leq 4.
\]
For $\ell_2$ less than 200, we have 
\[
\gen_{\cR}(X_{K_{\infty}, S}) =3 \ \ \mbox{if} \ \ \ell_2 = 19, 37, 73, 109, 163, 199
\]
and
\[
\gen_{\cR}(X_{K_{\infty}, S}) =4 \ \ \mbox{otherwise, namely if} \ \ \ell_2= 127, 181.
\]
Thus the above inequality is sharp in this case.
\item[(2)] Take $p=2$. Suppose that $\ell_1$, $\ell_2$, $\ell_3$ are three distinct primes such that $\ell_{i} \equiv 1$ (mod $4$). 
We take $K=\Q(\sqrt{\ell_1},\sqrt{\ell_2}, \sqrt{\ell_3})$ and $S=\{2, \ell_1, \ell_2, \ell_3, \infty\}$. 
Then $\Gal(K/\Q)\simeq (\Z/2\Z)^{\oplus 3}$, so $s=3$.
Taking account of the archimedean place, we know $t=1+3=4$.
Since $s(s+1)/2=6$, Theorem \ref{thm:bound1} says in this case 
\[
6 \leq \gen_{\cR}(X_{K_{\infty}, S}) \leq 10.
\]
Take $\ell_1=5$. 
For any $5< \ell_2 < \ell_3 \leq 100$, we have 
$
\gen_{\cR}(X_{K_{\infty}, S}) =7
$
except for 
\[
(\ell_2,\ell_3) = 
(17, 89), (37,41), (41,61), (41, 73), (41, 89), (53, 89), (73, 89), (89, 97),
\]
for which we have $\gen_{\cR}(X_{K_{\infty}, S}) =8$.
Also we have $\gen_{\cR}(X_{K_{\infty}, S}) =9$ for $(\ell_1,\ell_2,\ell_3)=(17,73,89)$, and
$\gen_{\cR}(X_{K_{\infty}, S}) =10$ for $(\ell_1,\ell_2,\ell_3)=(73,89,97)$. 
We do not have $\gen_{\cR}(X_{K_{\infty}, S}) =6$ at least in this range. 
\end{itemize}
\end{rem}

In this paper, we prove not only the above Theorem \ref{thm:bound1} but also its non-abelian generalization in Theorem \ref{thm:bound2}. 
We also study and determine the minimal numbers of generators and relations of the dual (Iwasawa adjoint) of $X_{K_{\infty}, S}$ (see Theorem \ref{thm:gen_dual}). 
This is relatively easier than Theorem \ref{thm:bound2}. 
Also, we give in \S \ref{ss:minusIM} some applications to the minus part of certain Iwasawa modules of CM-fields (see Corollary \ref{cor:bound4}), using Kummer duality. 

A key to the proof of our theorems is the existence of certain exact sequences, called Tate sequences. 
We remark here that Greither also used a different kind of Tate sequence in \cite{Grei13} to get information on the minimal numbers of generators of class groups of number fields. 
Our method of using the Tate sequences is totally different from Greither's.

This paper is organized as follows.
After algebraic preliminaries in \S \ref{sec:alg_prelim}, we will state the main results in \S \ref{sec:state}.
The proof is given in \S \S \ref{sec:Tate_arith}--\ref{sec:pf}.
Finally in \S \ref{sec:examples}, we will observe numerical examples.

\section*{Acknowledgments}

The authors would like to thank Yuta Nakamura, who computed $\gen_{\cR}(X_{K_{\infty}, S})$ for several examples in his master's thesis in a slightly different situation from ours.
They also thank Cornelius Greither heartily for his interest in the subject of this paper and for giving them some valuable comments.
The first and the second authors are supported by JSPS KAKENHI Grant Numbers 22K13898 and 22H01119, respectively.

\section{Algebraic preliminaries}\label{sec:alg_prelim}

\subsection{Minimal resolutions}\label{ss:general}

Let $R$ be a Noetherian local ring, which we do not assume to be commutative.
Let $\fm$ be the Jacobson radical of $R$, that is, $\fm$ is the maximal left (right) ideal of $R$.
For simplicity, let us assume that $\mathbf{k} := R/\fm$ is a commutative field.
We will often consider the case $R = \Z_p[[\cG]]$ for a pro-$p$ group $\cG$, in which case $R$ is indeed local and we have $\mathbf{k} = \F_p$ (see \cite[Proposition 5.2.16 (iii)]{NSW08}).

\begin{defn}
For a finitely generated (left) $R$-module $M$, we write $\gen_R(M)$ for the minimal number of generators of $M$ as an $R$-module.
Also, we write $r_R(M)$ for the minimal number of relations of $M$ as an $R$-module (see Definition \ref{defn:min_resol} below).
\end{defn}

\begin{rem}\label{rem:gen_pr}
The following observations will be often used.
\begin{itemize}
\item[(1)]
By Nakayama's lemma (e.g., \cite[Corollary 13.12]{Isa09}), we have
\[
\gen_R(M) = \gen_{R/I}((R/I) \otimes_R M)
\]
for any two-sided ideal $I \subset \fm$ of $R$.
In particular, we have
\[
\gen_R(M) = \gen_{\mathbf{k}}(\mathbf{k} \otimes_R M)
= \dim_{\mathbf{k}}(\mathbf{k} \otimes_R M).
\]
Therefore, for a finitely generated $\Z_p$-module $M$, we have
\[
\gen_{\Z_p}(M) = \dim_{\F_p}(\F_p \otimes_{\Z_p} M) = \rank_p(M).
\]
\item[(2)]
If we have an exact sequence
\[
0 \to M' \to M \to M'' \to 0
\]
of finitely generated $R$-modules, we have
\[
\gen_{R}(M'')
\leq \gen_{R}(M)
\leq \gen_{R}(M') + \gen_{R}(M'').
\]
The proof is standard.
\item[(3)]
In item (2) above, if we assume that $R$ is a discrete valuation ring (DVR), the formula is refined as
\[
\max \left\{ \gen_{R}(M'), \gen_{R}(M'') \right\}
\leq \gen_{R}(M)
\leq \gen_{R}(M') + \gen_{R}(M'').
\]
This follows from the structure theorem for finitely generated modules over principal ideal domains.
\end{itemize}
\end{rem}

\begin{eg}\label{eg:gen1}
Let us observe an example for which the formula in item (3) above does not hold when $R$ is not a DVR.
Let $R = \Z_p[[T]]$.
Consider $M = \Z_p[[T]]$ and its submodule
\[
M' = (p, T)^n = (p^n, p^{n-1}T, \dots, pT^{n-1}, T^n)
\]
with $n \geq 1$.
Then we have $\gen_{R}(M') = n + 1$ and $\gen_{R}(M) = 1$, so $\gen_{R}(M') \leq \gen_{R}(M)$ does not hold.
\end{eg}

Next we introduce the minimal resolutions of modules.

\begin{defn}\label{defn:min_resol}
Let $M$ be a finitely generated $R$-module.
We can construct an exact sequence of $R$-modules
\[
\cdots \to R^{r_2} \to R^{r_1} \to R^{r_0} \to M \to 0
\]
such that the image of each homomorphism $R^{r_{n+1}} \to R^{r_n}$ ($n \geq 0$) is contained in $\fm^{r_n}$.
Such a sequence is called a minimal resolution of $M$.
In this case, since $R^{r_{n+1}} \to R^{r_n}$ induces the zero map on $(R/\fm)^{r_{n+1}} \to (R/\fm)^{r_n}$, by the definition of the $\Tor$ functor, the integer $r_n$ coincides with
\[
r_n(M) = r_n^R(M) := \dim_{\mathbf{k}} \Tor_n^{R}(\mathbf{k}, M)
\]
for $n \geq 0$.
In particular, the integer $r_n$ is independent of the choice of minimal resolutions.
By definition we have 
\[
\gen_{R}(M)=r_0^R(M),
\qquad
r_R(M)=r_1^R(M).
\]
\end{defn}

\begin{lem}\label{lem:hom_r}
If $G$ is a finite $p$-group, we have
\[
r_n^{\Z_p[G]}(\Z_p) = \dim_{\F_p} H_n(G, \F_p).
\]
\end{lem}
\begin{proof}
This follows from $H_n(G, \F_p) \simeq \Tor_n^{\Z_p[G]}(\Z_p, \F_p) \simeq \Tor_n^{\Z_p[G]}(\F_p, \Z_p)$ and the formula in Definition \ref{defn:min_resol}.
\end{proof}

\subsection{Group homology}\label{ss:gp_hom}

In this subsection, we summarize facts about group homology.

Let $G$ be a finite group.
The following lemma is well-known.

\begin{lem}\label{lem:H3}
We have
\[
H_1(G, \Z) \simeq G^{\ab},
\]
the abelianization of $G$, and
\[
H_1(G, \Z/M\Z) \simeq G^{\ab}/M
\]
for $M \in \Z_{\geq 1}$.
\end{lem}

As for the second homology groups, if $G$ is abelian, it is known that $H_2(G, \Z)$ is isomorphic to $\bigwedge^2 G$ (see \cite[Chap.~V, Theorem 6.4 (iii)]{Bro82}).
If $G$ is not abelian, $H_2(G, \Z)$ is much harder to study, which is also known as the Schur multiplier of $G$ (cf.~\cite{Kar87}).

For now, we observe a relation between $H_n(G, \Z)$ and $H_n(G, \Z/M\Z)$ for a $p$-power $M$.

\begin{lem}\label{lem:H1}
Let $n \geq 2$.
For any $m \geq 1$, we have
\[
\rank_p H_n(G, \Z/p^m \Z) 
= \rank_p H_n(G, \Z) + \rank_p H_{n-1}(G, \Z).
\]
In particular, as the right hand side is independent from $m$, we have
\[
\rank_p H_n(G, \Z/p^m \Z) = \dim_{\F_p} H_n(G, \F_p).
\]
\end{lem}

\begin{proof}
This follows from the universal coefficient theorem (see \cite[Chap.~I, Proposition 0.8]{Bro82}, for example), 
which says in our case that 
$$
0 \to H_n(G, \Z)\otimes  \Z/p^m \Z \to H_n(G, \Z/p^m \Z) \to {\rm Tor}^{\Z}_1(H_{n-1}(G, \Z), \Z/p^m) \to 0
$$
is split exact.
\end{proof}

In case $G$ is abelian, it is not hard to compute the $p$-rank of the $n$-th homology group:

\begin{lem}\label{lem:H2}
Suppose $G$ is abelian and put $s = \rank_p G$.
Then we have
\[
\dim_{\F_p} H_n(G, \F_p) = \frac{s(s+1) \cdots (s+n-1)}{n!}
\]
for $n \geq 0$ (when $n = 0$, the right hand side is understood to be $1$).
\end{lem}

\begin{proof}
By replacing $G$ by its $p$-Sylow subgroup, we may assume that $G$ is a $p$-group.
As in \cite[\S 1.2]{GK15} or \cite[\S 4.3]{Kata_05}, we can construct an explicit minimal free resolution of $\Z_p$ as an $R$-module
$$ \cdots \to R^{s_3} \to R^{s_2} \to R^{s_1} \to R^{s_0} \to \Z_p \to 0$$
with $s_{n}=s(s+1) \cdots (s+n-1)/n!$.
Thus, the lemma follows from Lemma \ref{lem:hom_r}.
\end{proof}

We also need the following duality theorem between 
the cohomology groups and the homology groups (see \cite[Chap VI Proposition 7.1]{Bro82}, for example). 

\begin{lem}\label{lem:HN1}
Let $G$ be a finite group and $M$ a (discrete) $G$-module.
We define its Pontryagin dual $M^{\vee}$ by $M^{\vee}=\Hom(M, \Q/\Z)$. 
Then for any $n \in \Z_{\geq 0}$, we have an isomorphism between 
$H^{n}(G, M)$ and $\Hom(H_{n}(G,M^{\vee}), \Q/\Z)$. 
\end{lem}

\section{The main results}\label{sec:state}

\subsection{Setting}\label{ss:set}

As in \S \ref{sec:intro}, let $p$ be any prime number, $k$ a totally real field, and $k_{\infty}$ its cyclotomic $\Z_p$-extension.
For a finite set $S$ of places of $k$ such that $S$ contains all the archimedean places and all $p$-adic places, we write $M_{k, S}$ for the maximal abelian pro-$p$-extension of $k$ unramified outside $S$.

Let $K_{\infty}/k$ be a pro-$p$ Galois extension of totally real fields such that $K_{\infty}$ contains $k_{\infty}$ and the extension $K_{\infty}/k_{\infty}$ is finite.
We do not assume that $K_{\infty}/k$ is abelian, but we have to assume the following.

\begin{ass}\label{ass:spl}
There exists an intermediate finite Galois extension $K/k$ of $K_{\infty}/k$ such that 
\[
K_{\infty} = k_{\infty} K,
\quad
k_{\infty} \cap K = k.
\]
In other words, the map induced by the restriction maps
\[
\Gal(K_{\infty}/k) \to \Gal(k_{\infty}/k) \times \Gal(K/k)
\]
 is an isomorphism.
\end{ass}

\begin{lem}\label{lem:ab_ass}
If $K_{\infty}/k$ is abelian, then Assumption \ref{ass:spl} holds.
\end{lem}

\begin{proof}
We consider the restriction homomorphism 
$f: \Gal(K_{\infty}/k) \twoheadrightarrow \Gal(k_{\infty}/k)$.
Since $f$ is a homomorphism of $\Z_p$-modules and the target is free, 
$f$ has a section. 
We take a section and define $K$ to be the fixed field of the image of the section.
A point is that $K/k$ is then automatically Galois as $K_{\infty}/k$ is abelian.
\end{proof}

Set $\cG = \Gal(K_{\infty}/k)$ and $G = \Gal(K_{\infty}/k_{\infty})$.
We take an $S$ such that $K_{\infty}/k$ is unramified outside $S$.
Let $M_{K_{\infty}, S}/K_{\infty}$ be the Galois group of the maximal abelian pro-$p$-extension of $K_{\infty}$ that is unramified outside places lying above $S$, and 
$X_{K_{\infty}, S} = \Gal(M_{K_{\infty}, S}/K_{\infty})$ 
as in the Introduction.
Then it is known that $X_{K_{\infty}, S}$ is a finitely generated torsion module over the associated Iwasawa algebra $\cR = \Z_p[[\cG]]$.
Since $K_{\infty}/k$ is a pro-$p$ extension, the algebra $\cR$ is a local ring whose residue field is $\F_p$.

\subsection{The statements}\label{ss:state}

We use the notation in \S \ref{ss:set}.
To state the result, let us put
\[
s_n = \dim_{\F_p} H_n(G, \F_p)
\]
for $n \geq 0$ (recall $G = \Gal(K_{\infty}/k_{\infty})$).
For instance, we have $s_0 = 1$ and $s_1 = \rank_p G^{\ab}$ by Lemma \ref{lem:H3}.
Recall that Lemma \ref{lem:H2} tells us an explicit formula of $s_n$ in case $K_{\infty}/k_{\infty}$ is {\it abelian}; in particular, we have $s_2 = s(s+1)/2$ and $s_3 = s(s+1)(s+2)/6$ with $s = \rank_p G (= s_1)$.

The following is the main result, which contains a non-abelian generalization of Theorem \ref{thm:bound1}.

\begin{thm}\label{thm:bound2}
When Assumption \ref{ass:spl} is satisfied, the following inequalities and equalities hold.
\begin{itemize}
\item[(1)]
Put $t = \rank_p \Gal(M_{k, S}/M_{k, S} \cap K_{\infty})$.
Then we have
\[
\max \left\{ s_2, t \right\} 
\leq \gen_{\cR}(X_{K_{\infty}, S})
\leq s_2 + t.
\]
\item[(2)]
We have
\[
r_n(X_{K_{\infty}, S}) 
= s_{n+2} + s_{n+1}
\]
for $n \geq 2$ and
\[
r_1(X_{K_{\infty}, S}) - r_0(X_{K_{\infty}, S}) 
= r_{\cR}(X_{K_{\infty}, S}) - \gen_{\cR}(X_{K_{\infty}, S}) = s_3.
\]
\end{itemize}
\end{thm}

It is easy to see that Theorem \ref{thm:bound2} implies Theorem \ref{thm:bound1}, thanks to Lemma \ref{lem:ab_ass}.

\vspace{5mm}

We also prove corresponding theorems concerning the dual (Iwasawa adjoint) of $X_{K_{\infty}, S}$. 
For a finitely generated torsion $\cR$-module $M$, we define 
the dual (Iwasawa adjoint) of $M$ by 
$$
M^{*} = \Ext^1_{\cR}(M, \cR).
$$
Put $\Lambda = \Z_p[[\Gal(K_{\infty}/K)]]$ by using Assumption \ref{ass:spl}, so $\cR$ is isomorphic to $\Lambda[G]$.
Then we have
$$
M^{*} \simeq \Ext^1_{\Lambda}(M, \Lambda)
$$
because $\Hom_{\cR}(N, \cR) \simeq \Hom_{\Lambda}(N, \Lambda)$ for any $\cR$-module $N$.
Therefore, our $M^{*}$ coincides with the Iwasawa adjoint of $M$ in \cite[Definition 5.5.5]{NSW08}, \cite[\S 5.1]{Iwa59}, \cite[\S 1.3]{Iwa73}. 

We are interested in the $\cR$-module $X_{K_{\infty}, S}^{*}$.
It is known that the structure of $X_{K_{\infty}, S}^{*}$ is often simpler than $X_{K_{\infty}, S}$ itself (e.g., when we are concerned with their Fitting ideals).
The following theorem implies that we encounter such a phenomenon when we are concerned with the minimal resolutions.

\begin{thm}\label{thm:gen_dual}
When Assumption \ref{ass:spl} is satisfied, the following equalities hold.
\begin{itemize}
\item[(1)]
We have
\[
\gen_{\cR}(X_{K_{\infty}, S}^{*})
= \rank_p \Gal(M_{k, S}/k_{\infty}).
\]
\item[(2)]
If $K_{\infty} \supsetneqq k_{\infty}$, then we have
\[
r_n(X_{K_{\infty}, S}^{*}) 
= s_{n-2} + s_{n-3}
\]
for $n \geq 3$,
\[
r_2(X_{K_{\infty}, S}^{*}) = s_0 + s_0 = 2,
\]
and
\[
r_1(X_{K_{\infty}, S}^{*}) - r_0(X_{K_{\infty}, S}^{*})
=r_{\cR}(X_{K_{\infty}, S}^{*}) - \gen_{\cR}(X_{K_{\infty}, S}^{*})
= s_0 = 1.
\]
If $K_{\infty} = k_{\infty}$, then we have $r_n(X_{K_{\infty}, S}^{*}) = 0$ for $n \geq 2$ and $r_1(X_{K_{\infty}, S}^{*}) - r_0(X_{K_{\infty}, S}^{*}) = 0$.
\end{itemize}
\end{thm}

In \S \ref{sec:abst_Tate}, we will prove $s_2 \leq \gen_{\cR}(X_{K_{\infty, S}})$ in Theorem \ref{thm:bound2}(1), Theorem \ref{thm:bound2}(2), and Theorem \ref{thm:gen_dual}(2).
These parts follow only from the existence of the Tate sequence introduced in \S \ref{sec:Tate_arith}.
The rest of the statements ($t \leq \gen_{\cR}(X_{K_{\infty}, S}) \leq s_2 + t$ in Theorem \ref{thm:bound2}(1) and Theorem \ref{thm:gen_dual}(1)) will be proved in \S \ref{sec:pf}.

\subsection{Applications for the minus parts of Iwasawa modules for CM-extensions}\label{ss:minusIM}

In this subsection we apply the main theorems in the previous subsection to CM-extensions. 
We keep the notation in \S \ref{ss:set}, so $K_{\infty}/k$ is an extension of totally real fields satisfying Assumption \ref{ass:spl}.
{\it Only in this subsection} we assume that $p$ is odd, which is mainly for making the functor of taking the character component exact for characters of $\Gal(K_{\infty}(\mu_{p})/K_{\infty})$.

We consider the field ${\mathcal K}_{\infty}=K_{\infty}(\mu_{p})$ obtained by adjoining all $p$-th roots of unity to $K_{\infty}$. 
So ${\mathcal K}_{\infty}/k$ is a CM-extension. 
We also use an intermediate field ${\mathcal K}_{n}$ of the $\Z_p$-extension ${\mathcal K}_{\infty}/K(\mu_{p})$ 
such that $[{\mathcal K}_{n}:K(\mu_{p})]=p^n$ for each $n \geq 0$.
Let ${\mathcal L}_{n}$ be the maximal abelian pro-$p$-extension of ${\mathcal K}_{n}$ unramified {\it everywhere}.
So $\Gal({\mathcal L}_{n}/{\mathcal K}_{n})$ is isomorphic to the $p$-component $A_{{\mathcal K}_{n}}$ of the ideal class group of ${\mathcal K}_{n}$ by class field theory. 
We denote by $A_{{\mathcal K}_{\infty}}$ the inductive limit of $A_{{\mathcal K}_{n}}$, which is a discrete $\Z_p[[\Gal({\mathcal K}_{\infty}/k)]]$-module. 
We write ${\mathcal X}_{{\mathcal K}_{\infty}}$ for the projective limit of $A_{{\mathcal K}_{n}}$. 
Then ${\mathcal X}_{{\mathcal K}_{\infty}}$ is a compact $\Z_p[[\Gal({\mathcal K}_{\infty}/k)]]$-module.
Defining ${\mathcal L}_{\infty}$ to be the maximal abelian pro-$p$-extension of ${\mathcal K}_{\infty}$ unramified everywhere, we know that ${\mathcal X}_{{\mathcal K}_{\infty}}=\Gal({\mathcal L}_{\infty}/{\mathcal K}_{\infty})$.
Let $n_{0}$ be the smallest integer such that all  $p$-adic places are totally ramified in ${\mathcal K}_{\infty}/{\mathcal K}_{n_0}$. 
We define a submodule ${\mathcal Y}_{{\mathcal K}_{\infty}}$ of ${\mathcal X}_{{\mathcal K}_{\infty}}$ by 
${\mathcal Y}_{{\mathcal K}_{\infty}}=\Gal({\mathcal L}_{\infty}/{\mathcal L}_{n_0}{\mathcal K}_{\infty})$.

Put $\Delta=\Gal({\mathcal K}_{\infty}/K_{\infty})$, which is of order prime to $p$ by our assumption $p \neq 2$ 
in this subsection. 
Therefore, since $\Gal({\mathcal K}_{\infty}/k)\simeq \cG \times \Delta$, 
any $\Z_p[[\Gal({\mathcal K}_{\infty}/k)]]$-module $M$ is decomposed into 
$M=\bigoplus_{\chi} M^{\chi}$ where 
$\chi$ runs over all characters of $\Delta$ with values in $\Z_p^{\times}$, and  $M^{\chi}$ is the $\chi$-component of $M$ defined by 
$$M^{\chi}=\{x \in M \mid \sigma(x)=\chi(\sigma)x \ \mbox{for any} \ \sigma \in \Delta\}.$$
Note that each $M^{\chi}$ is an $\cR$-module. 
Let $\omega: \Delta \to \Z_p^{\times}$ be the Teichm\"{u}ller character, giving the action on $\mu_{p}$. 
Using our main results in \S \ref{ss:state}, we study $A_{{\mathcal K}_{\infty}}^{\omega}$ and ${\mathcal Y}_{{\mathcal K}_{\infty}}^{\omega}$.
Note that $\omega$ is an odd character, so the complex conjugation acts on these modules as $-1$.

Let $S_{p}$ be the set of all $p$-adic places and all archimedean places. 
Recall that we write $(-)^{\vee}$ for the Pontryagin dual.
By Kummer pairing (see \cite[Theorem 11.4.3]{NSW08} or \cite[Proposition 13.32]{Was97}), we have an isomorphism 
$$
(A_{{\mathcal K}_{\infty}}^{\omega})^{\vee}(1) \simeq X_{K_{\infty},S_{p}},
$$
where $(1)$ is Tate twist.
Also, by \cite[Theorem 11.1.8]{NSW08} we have 
$$
(A_{{\mathcal K}_{\infty}})^{\vee}
 \simeq {\mathcal Y}_{{\mathcal K}_{\infty}}^{*}.
$$
Therefore, we have 
$$X_{K_{\infty},S_{p}}^{*}(1) 
\simeq 
(A_{{\mathcal K}_{\infty}}^{\omega})^{\vee}(1)^*(1)
\simeq 
\left((A_{{\mathcal K}_{\infty}}^{\omega})^{\vee}\right)^*
\simeq 
(({\mathcal Y}_{{\mathcal K}_{\infty}}^{\omega})^{*})^{*}.
$$
For any finitely generated torsion $\cR$-module $M$ which has no nontrivial finite submodule, 
we know $(M^{*})^{*} \simeq M$ (see, for example, \cite[Proposition 5.5.8 (iv)]{NSW08}). 
Since ${\mathcal X}_{{\mathcal K}_{\infty},S_{p}}$ has no nontrivial finite submodule, so does ${\mathcal Y}_{{\mathcal K}_{\infty}}^{\omega}$.
Therefore, it follows from the previous isomorphism that 
$$X_{K_{\infty},S_{p}}^{*} \simeq {\mathcal Y}_{{\mathcal K}_{\infty}}^{\omega}(-1)$$
is an isomorphism. 

Thus, from Theorems \ref{thm:bound2} and \ref{thm:gen_dual} we get

\begin{cor}\label{cor:bound4}
In Theorem \ref{thm:bound2}, we further assume that $p >2$ and $S = S_p$ (so $K_{\infty}/k_{\infty}$ is unramified outside $p$). 
\begin{itemize}
\item[(1)] 
Then we have 
$$\max \left\{ s_2, t \right\} 
\leq \gen_{\cR}((A_{{\mathcal K}_{\infty}}^{\omega})^{\vee})
\leq s_2 + t$$ 
and 
$$
r_{\cR}((A_{{\mathcal K}_{\infty}}^{\omega})^{\vee}) = \gen_{\cR}((A_{{\mathcal K}_{\infty}}^{\omega})^{\vee}) + s_3.
$$
\item[(2)]
Put $t'=\rank_p \Gal(M_{k, S}/k_{\infty})$.
Then we have 
$\gen_{\cR}({\mathcal Y}_{{\mathcal K}_{\infty}}^{\omega})
= t'$ and 
$$
r_{\cR}({\mathcal Y}_{{\mathcal K}_{\infty}}^{\omega}) = 
\left\{
\begin{array}{ll}
t'+1 & \ \mbox{if $K_{\infty} \supsetneqq k_{\infty}$} \\
t' & \ \mbox{if $K_{\infty} = k_{\infty}$.}
\end{array}\right.
$$
 \end{itemize}
\end{cor}

\section{The Tate sequence}\label{sec:Tate_arith}

A key ingredient to prove Theorems \ref{thm:bound2} and \ref{thm:gen_dual} is an exact sequence that $X_{K_{\infty}, S}$ satisfies, which is often called the Tate sequence.
Indeed, as noted in the final paragraph of \S \ref{ss:state}, parts of main theorems can be deduced from the existence of the Tate sequence only.
On the other hand, the other parts require additional arithmetic study that we will do in \S \ref{sec:pf}.
The Tate sequence also played a key role in computing the Fitting ideal of $X_{K_{\infty}, S}$ in the work \cite{GK15}, \cite{GK17}, and \cite{GKT20} that we mentioned in \S \ref{sec:intro}.

In order to prove the main theorems, we need the Tate sequence of the following type.

\begin{thm}\label{thm:Tate}
There exists an exact sequence of $\cR$-modules
\[
0 \to X_{K_{\infty}, S} \to P \overset{\phi}{\to} Q \to \Z_p \to 0,
\]
where $P$ and $Q$ are finitely generated torsion $\cR$-modules whose projective dimensions are $\leq 1$.
Moreover, this sequence is functorial when $K_{\infty}$ varies.
More precisely, for a finite normal subgroup $H$ of $\Gal(K_{\infty}/k)$, we have an exact sequence 
\[
0 \to X_{K_{\infty}^H, S} \to P_H \overset{\phi_H}{\to} Q_H \to \Z_p \to 0
\]
over the Iwasawa algebra $\Z_p[[\Gal(K_{\infty}^H/k)]]$,
where $P_H$, $Q_H$ denote the $H$-coinvariant modules, and $\phi_H$ the homomorphism induced by $\phi$.
\end{thm}

\begin{proof}
For an intermediate field $F$ of $K_{\infty}/k$ with $[F:k]<\infty$, 
we use a perfect complex $\RG_{c}({\mathcal O}_{F,S}, \Z_p(1))$ 
in Burns--Flach \cite[Proposition 1.20]{BF96}. 
Note that this complex works well even for $p=2$ (here, we use our assumption that $S$ contains all archimedean places).
Taking the project limit, we get a perfect complex 
$C^{\bullet}=\RG_{c}({\mathcal O}_{K_{\infty},S}, \Z_p(1))$ 
which is quasi-isomorphic to a complex of the form $[C^0 \overset{d^0}{\to} C^1 \overset{d^1}{\to} C^2]$ concentrated on degrees $0$, $1$, $2$ with $C^i$ finitely generated projective over $\cR$ and
whose cohomology groups are 
$$H^1(C^{\bullet})=H^1(O_{K_{\infty},S}, \Q_p/\Z_p)^{\vee}=X_{K_{\infty},S}, \ \ 
H^2(C^{\bullet})=H^0(O_{K_{\infty},S}, \Q_p/\Z_p)^{\vee}=\Z_p,
$$
(see \cite[page 86, line 6]{BF96}) and $H^i(C^{\bullet})=0$ for $i \neq 1$, $2$,
where we used the weak Leopoldt conjecture which is proven in this case (see \cite[Theorem 10.3.25]{NSW08}).
By the definition of cohomology groups, we have an exact sequence
\[
0 \to X_{K_{\infty},S} \to C^1/C^0 \overset{\Phi}{\to} C^2 \to \Z_p \to 0,
\]
where we regard $C^0$ as a submodule of $C^1$ via $d^0$ and $\Phi$ is induced by $d^1$.
Take a non-zero-divisor $f$ in the center of $\cR$ that annihilates $\Z_p$.
Then the image of $\Phi$ contains $f C^2$, so by the projectivity of $C^2$ we can construct a commutative diagram of $\cR$-modules
\[
\xymatrix{
	C^2 \ar@{=}[r] \ar@{^(->}[d]
	& C^2 \ar@{^(->}[d]^{f \times}\\
	C^1/C^0 \ar[r]_-{\Phi}
	& C^2.
}
\]
Then defining $P$ and $Q$ as the cokernel of these vertical maps respectively, we obtain the Tate sequence as claimed.
The functoriality follows from that of $\RG_{c}({\mathcal O}_{K_{\infty},S}, \Z_p(1))$.
\end{proof} 

\section{Abstract Tate sequences}\label{sec:abst_Tate}

Set $\Lambda = \Z_p[[T]]$.
Let $G$ be a (not necessarily abelian) finite $p$-group.

Motivated by Theorem \ref{thm:Tate}, we study a $\Lambda[G]$-module $X$ that satisfies an abstract Tate sequence, that is:

\begin{setting}\label{set:Tate}
There exists an exact sequence of $\Lambda[G]$-modules
\begin{equation}\label{eq:Tate}
0 \to X \to P \overset{\phi}{\to} Q \to \Z_p \to 0,
\end{equation}
where both $P$ and $Q$ are finitely generated torsion $\Lambda[G]$-modules whose projective dimensions are $\leq 1$.
\end{setting}

In this section, we show that the existence of a Tate sequence gives a severe constraint on the integers $\gen_{\Lambda[G]}(X) = r_0(X)$, $r_{\Lambda[G]}(X) = r_1(X)$, and $r_n(X)$ ($n \geq 2$).

\subsection{The statements}\label{ss:state_abst}

To state the result, let us define
\[
s_n = \dim_{\F_p} H_n(G, \F_p)
\]
for $n \geq 0$.

The following are the main theorems in this section.
As noted in the final paragraph of \S \ref{ss:state}, those are enough to show parts of Theorems \ref{thm:bound2} and \ref{thm:gen_dual}.

\begin{thm}\label{thm:min_resol}
Let $X$ be a $\Lambda[G]$-module that satisfies a Tate sequence as in Setting \ref{set:Tate}.
\begin{itemize}
\item[(1)]
We have $\gen_{\Lambda[G]}(X) \geq s_2$.
\item[(2)]
We have
\[
r_n(X) 
= s_{n+2} + s_{n+1}
\]
for $n \geq 2$ and
\[
r_1(X) - r_0(X) 
= s_3.
\]
\end{itemize}
\end{thm}

Claims (1) and (2) will be proved respectively in \S \ref{ss:pf_spec} and \S \ref{ss:resol}.

For a $\Lambda[G]$-module $M$, we define its dual (Iwasawa adjoint) by
\[
M^{*} = \Ext^1_{\Lambda[G]}(M, \Lambda[G]) 
\simeq \Ext^1_{\Lambda}(M, \Lambda).
\]
The corresponding theorem for the dual is:

\begin{thm}\label{thm:min_resol2}
Let $X$ be a $\Lambda[G]$-module that satisfies a Tate sequence as in Setting \ref{set:Tate}.
\begin{itemize}
\item[(1)]
We have $\gen_{\Lambda[G]}(X^{*}) \geq s_1$.
\item[(2)]
If $G$ is non-trivial, then we have
\[
r_n(X^{*}) 
= s_{n-2} + s_{n-3}
\]
for $n \geq 3$,
\[
r_2(X^{*}) = s_0 + s_0 (= 2),
\]
and
\[
r_1(X^{*}) - r_0(X^{*})
= s_0 (= 1).
\]
If $G$ is trivial, then we have $r_n(X^{*}) = 0$ for $n \geq 2$ and $r_1(X^{*}) - r_0(X^{*}) = 0$.
\end{itemize}
\end{thm}

This theorem will be proved in \S \ref{ss:resol2}.
The idea is basically the same as that of Theorem \ref{thm:min_resol}.
However, we need an additional algebraic proposition shown in \S \ref{ss:equiv}.

\subsection{Specialization}\label{ss:spec}

We consider modules over $\Lambda = \Z_p[[T]]$.
As explained in Example \ref{eg:gen1}, $\gen_{\Lambda}(-)$ does not behave very well for short exact sequences.
A key idea to prove the main theorems is to apply specialization method to reduce to modules over DVRs.

We define
\[
\cF = \{p\} \cup \{ f \in \Z_p[T] \mid \text{$f$ is an irreducible monic distinguished polynomial} \}.
\]
Here, a monic distinguished polynomial is by definition a polynomial of the form
\[
T^e + a_1 T^{e-1} + \cdots + a_e,
\]
where $a_1, \dots, a_e \in p\Z_p$.
By the Weierstrass preparation theorem, any prime element of $\Lambda$ can be written as the product of a unit element and an element of $\cF$ in a unique way.

For each $f \in \cF$, put
\[
\cO_f = \Lambda/(f),
\]
which is a domain.
We define a subset $\cF_0 \subset \cF$ by
\[
\cF_0 = \{ f \in \cF \mid \text{$\cO_f$ is a DVR} \}.
\]

The following lemma tells us a concrete description of $\cF_0$.
Although the lemma is unnecessary for the proof of the main results, we include it in this paper to clarify the situation.

\begin{lem}
We have $\cF_0 = \{p\} \cup \cF_1 \cup \cF_2$, where we put
\[
\cF_1 = 
\{T - \alpha \mid \alpha \in p\Z_p\}
\]
and
\[
\cF_2 = \{T^e + a_1 T^{e-1} + \dots + a_e \mid e \geq 2, a_1, \dots, a_{e-1} \in p\Z_p, a_e \in p\Z_p \setminus p^2\Z_p\}.
\]
\end{lem}

\begin{proof}
It is clear that $\{p \} \subset \cF_0$ and $\cF_1 \subset \cF_0$.
Also, $\cF_2 \subset \cF_0$ holds by the Eisenstein irreducibility criterion.
Therefore, it remains to only show $\cF_0 \setminus \{p\} \subset \cF_1 \cup \cF_2$.

Let $f \in \cF_0 \setminus \{p\}$.
Since $\cO_f$ is a DVR, it is the integral closure of $\Z_p$ in the $p$-adic field $K_f = \Frac(\cO_f)$.
Moreover, since the residue field of $\cO_f$ is the same as that of $\Lambda$, namely $\F_p$, we see that the extension $K_f/\Q_p$ is totally ramified.
In case the extension $K_f/\Q_p$ is trivial, we have $\deg(f) = 1$, so we obtain $f \in \cF_1$.
In case $K_f/\Q_p$ is non-trivial, the image of $T$ in $\cO_f$ must be a uniformizer of $\cO_f$, so its minimal polynomial $f$ is in $\cF_2$ (see \cite[Chap.~I, Proposition 18]{Ser79}).
This completes the proof.
\end{proof}

\subsection{Proof of Theorem \ref{thm:min_resol}(1)}\label{ss:pf_spec}

Let us now study a $\Lambda[G]$-module $X$ satisfying a Tate sequence as in Setting \ref{set:Tate}.
We define a $\Lambda$-module $X_{(G)}$ by
\begin{equation}\label{eq:Tate2}
0 \to X_{(G)} \to P_G \overset{\phi_G}{\to} Q_G \to \Z_p \to 0,
\end{equation}
where $P_G$ and $Q_G$ denote the $G$-coinvariant modules and $\phi_G$ denotes the induced homomorphism.
Note that $X_{(G)}$ does not coincide with the coinvariant module $X_G$ in general; in fact, the difference is what we shall investigate from now on.

The following proposition is a key to prove the main theorem.

\begin{prop}\label{prop:sp_hom}
Let $f \in \cF$ be an element that is prime to both $\ch_{\Lambda}(P)$ and $\ch_{\Lambda}(Q)$, where $\ch_{\Lambda}(-)$ denotes the characteristic polynomial.
We set $m = \ord_p(f(0)) \geq 1$.
Then we have an exact sequence of finitely generated torsion $\cO_f$-modules
\[
0 \to H_2(G, \Z_p/p^m \Z_p) \to (X/f)_G \to X_{(G)}/f \overset{\pi}{\to} H_1(G, \Z_p/p^m \Z_p) \to 0.
\]
\end{prop}

\begin{proof}
Firstly note that $f(0) \neq 0$ since $\ch_{\Lambda}(Q)$ is divisible by $\ch_{\Lambda}(\Z_p) = (T)$.
By taking modulo $f$ of the sequence \eqref{eq:Tate}, we obtain an exact sequence of finitely generated torsion $\cO_{f}[G]$-modules
\[
0 \to X/f \to P/f \overset{\ol{\phi}}{\to} Q/f \to \Z_p/p^m \Z_p \to 0.
\]
Let $L$ denote the image of the map $\ol{\phi}: P/f \to Q/f$.
Since both $P/f$ and $Q/f$ are $G$-cohomologically trivial, taking the $G$-homology, we obtain exact sequences
\[
0 \to H_1(G, L) 
\to (X/f)_G 
\to (P/f)_G 
\to L_G 
\to 0
\]
and
\[
0 \to H_1(G, \Z_p/p^m \Z_p)
\to L_G
\to Q_G/f
\to \Z_p/p^m \Z_p
\to 0
\]
and also an isomorphism $H_2(G, \Z_p/p^m \Z_p) \simeq H_1(G, L)$.

We can combine these observations with the exact sequence obtained by taking modulo $f$ of sequence \eqref{eq:Tate2} to construct a diagram
\[
\xymatrix{
	& & & & H_1(G, \Z_p/p^m \Z_p) \ar@{^(->}[d] & \\
	0 \ar[r]
	& H_2(G, \Z_p/p^m \Z_p) \ar[r]
	& (X/f)_G \ar[r] \ar[d]
	& (P/f)_G \ar[r] \ar[d]_{\simeq}
	& L_G \ar[r] \ar[d]
	& 0\\
	& 0 \ar[r]
	& X_{(G)}/f \ar[r]
	& P_G/f \ar[r]
	& Q_G/f \ar[r] \ar@{->>}[d]
	& \Z_p/p^m \Z_p \ar[r]
	& 0 \\
	& & & & \Z_p/p^m \Z_p & &
}
\]
This is a commutative diagram of finitely generated torsion $\cO_f$-modules.
By applying the snake lemma, we obtain the proposition.
\end{proof}

\begin{proof}[Proof of Theorem \ref{thm:min_resol}(1)]
In Proposition \ref{prop:sp_hom}, we take $f$ so that $f \in \cF_0$, i.e., $\cO_f$ is a DVR.
Then the injective homomorphism from $H_2(G, \Z_p/p^m \Z_p)$ to $(X/f)_G$ in Proposition \ref{prop:sp_hom} implies
\[
\gen_{\Lambda[G]}(X)
= \gen_{\cO_f}((X/f)_G) 
\geq \gen_{\cO_f}(H_2(G, \Z_p/p^m \Z_p)),
\]
where the first equality follows from Nakayama's lemma.
Since $H_2(G, \Z_p/p^m \Z_p)$ is annihilated by $T$ and $\cO_f/(T) \simeq \Z_p/p^m\Z_p$, we have 
\[
\gen_{\cO_f}(H_2(G, \Z_p/p^m \Z_p))
= \gen_{\Z_p/p^m\Z_p}(H_2(G, \Z_p/p^m \Z_p))
= \dim_{\F_p}(H_2(G, \F_p))
= s_2,
\]
where the second equality follows from Lemma \ref{lem:H1}.
Combining these formulas, we obtain $\gen_{\Lambda[G]}(X) \geq s_2$, as claimed.
\end{proof}

\begin{rem}
This argument also shows
\[
\max \left\{ s_2, \gen_{\cO_f}(\Ker(\pi)) \right\} 
\leq \gen_{\Lambda[G]}(X)
\leq s_2 + \gen_{\cO_f}(\Ker(\pi))
\]
with $\pi$ as in Proposition \ref{prop:sp_hom}.
Then the full statement of Theorem \ref{thm:bound2}(1) follows if we can take $f$ so that $\gen_{\cO_f}(\Ker(\pi))$ coincides with the $t$ in Theorem \ref{thm:bound2}(1).
This is indeed possible, but in \S \ref{sec:pf} we will give a more direct proof for the rest of Theorem \ref{thm:bound2}(1) instead.
\end{rem}

\subsection{Proof of Theorem \ref{thm:min_resol}(2)}\label{ss:resol}

We introduce several lemmas.
We abbreviate $r_n^{\Lambda[G]}(-)$ as $r_n(-)$ and we never omit the coefficient ring otherwise.

\begin{lem}\label{lem:Fitt_inv1}
Let $P$ be a finitely generated torsion $\Lambda[G]$-module whose projective dimension is $\leq 1$.
Then we have
$r_n(P) = 0$ for $n \geq 2$ and $r_1(P) = r_0(P)$.
\end{lem}

\begin{proof}
By the assumption on $P$, there exists a presentation of $P$ of the form $0 \to \Lambda[G]^a \to \Lambda[G]^a \to P \to 0$.
The lemma follows immediately from this.
\end{proof}

\begin{lem}\label{lem:rn_rel}
Let $0 \to M' \to P \to M \to 0$ be a short exact sequence of finitely generated torsion $\Lambda[G]$-modules such that the projective dimension of $P$ is $\leq 1$.
Then we have
\[
r_n(M') = r_{n+1}(M)
\]
for $n \geq 2$ and
\[
r_1(M') - r_0(M') = r_2(M) - r_1(M) + r_0(M).
\]
\end{lem}

\begin{proof}
This follows immediately from the long exact sequence of $\Tor_*^{\Lambda[G]}(\F_p, -)$ applied to the given sequence, taking Lemma \ref{lem:Fitt_inv1} into account.
\end{proof}

\begin{lem}\label{lem:resol_bc}
Let $M$ be a finitely generated $\Z_p[G]$-module that is free over $\Z_p$.
We regard $M$ as a $\Lambda[G]$-module so that $T$ acts trivially on $M$.
Then we have
\[
r_n^{\Lambda[G]}(M) = r_n^{\Z_p[G]}(M) + r_{n-1}^{\Z_p[G]}(M)
\]
for $n \geq 0$.
Here, we set $r_{-1}^{\Z_p[G]}(M) = 0$.
\end{lem}

\begin{proof}
Let us take a minimal resolution of $M$ as a $\Z_p[G]$-module
\[
\cdots \to \Z_p[G]^{\rho_2} \to \Z_p[G]^{\rho_1} \to \Z_p[G]^{\rho_0} \to M \to 0,
\]
where we put $\rho_n = r_n^{\Z_p[G]}(M)$.
We have an exact sequence
\[
0 \to \Lambda \overset{\times T}{\to} \Lambda \to \Z_p \to 0,
\]
which may be regarded as a minimal resolution of $\Z_p$ as a $\Lambda$-module.
Then we take the tensor product over $\Z_p$ of the two complex above (omitting $M$ and $\Z_p$ respectively).
As a result, we obtain an exact sequence
\[
\cdots \to \Lambda[G]^{\rho_2 + \rho_1} \to \Lambda[G]^{\rho_1 + \rho_0} \to \Lambda[G]^{\rho_0} \to M \to 0.
\]
By construction, this is a minimal resolution of $M$ as a $\Lambda[G]$-module.
This completes the proof of the lemma.
\end{proof}

Now we are ready to prove Theorem \ref{thm:min_resol}(2).

\begin{proof}[Proof of Theorem \ref{thm:min_resol}(2)]
First we recall $s_n = r_n^{\Z_p[G]}(\Z_p)$ by Lemma \ref{lem:hom_r}.
By applying Lemma \ref{lem:resol_bc} to $M = \Z_p$, we obtain
\[
r_n(\Z_p) = r_n^{\Lambda[G]}(\Z_p) = s_n + s_{n-1}
\]
for $n \geq 0$, where we set $s_{-1} = 0$.
We apply Lemma \ref{lem:rn_rel} to the two short exact sequences obtained by splitting the Tate sequence.
As a consequence, we obtain
\[
r_n(X) = r_{n+2}(\Z_p) = s_{n+2} + s_{n+1}
\]
for $n \geq 2$ and
\begin{align}
r_1(X) - r_0(X)
& = r_3(\Z_p) - r_2(\Z_p) + r_1(\Z_p) - r_0(\Z_p)\\
& = (s_3 + s_2) - (s_2 + s_1) + (s_1 + s_0) - (s_0 + s_{-1})\\
& = s_3.
\end{align}
This completes the proof.
\end{proof}

\subsection{An algebraic proposition}\label{ss:equiv}

This subsection provides preliminaries to the proof of Theorem \ref{thm:min_resol2}.
Let $\cC$ be the category of finitely generated torsion $\Lambda[G]$-modules whose projective dimension over $\Lambda$ is $\leq 1$, that is, those that do not have nontrivial finite submodules.

We also write $\cP$ for the subcategory of $\cC$ that consists of modules whose projective dimension over $\Lambda[G]$ is $\leq 1$.

For a module $M \in \cC$, it is known that the dual
\[
M^{*} = \Ext^1_{\Lambda[G]}(M, \Lambda[G]) 
\simeq \Ext^1_{\Lambda}(M, \Lambda)
\]
is also in $\cC$ and $(M^*)^* \simeq M$ (\cite[Propositions 5.5.3 (ii) and 5.5.8 (iv)]{NSW08}).
Moreover, if $P \in \cP$, we have $P^{*} \in \cP$.
These facts are also explained in \cite[\S 3.1]{Kata_05}.

In this subsection, we prove the following proposition.

\begin{prop}\label{prop:stab}
Let $d \geq 0$ be an integer.
Let us consider exact sequences
\[
0 \to N \to P_1 \to \cdots \to P_d \to M \to 0
\]
and
\[
0 \to N' \to P_1' \to \cdots \to P_d' \to M' \to 0
\]
in $\cC$ such that $P_i, P_i' \in \cP$ for $1 \leq i \leq d$.
Then the following hold.
\begin{itemize}
\item[(1)]
If $M \simeq M'$, then we have
\[
r_n(N) = r_n(N')
\]
for $n \geq 2$ and
\[
r_1(N) - r_0(N) = r_1(N') - r_0(N').
\]
\item[(2)]
Similarly, if $N \simeq N'$, then we have
\[
r_n(M) = r_n(M')
\]
for $n \geq 2$ and
\[
r_1(M) - r_0(M) = r_1(M') - r_0(M').
\]
\end{itemize}
\end{prop}

\begin{rem}\label{rem:resol_indep}
It is easy to deduce claim (1) from Lemma \ref{lem:rn_rel}.
Indeed, we have 
\[
r_n(N) = r_{n+d}(M)
\]
for $n \geq 2$ and 
\[
r_1(N) - r_0(N) = \sum_{i=0}^{d+1} (-1)^{i} r_{d+1-i}(M).
\]
On the other hand, claim (2) cannot be deduced from Lemma \ref{lem:rn_rel}.
Roughly speaking, claims (1) and (2) are respectively what we need to prove Theorems \ref{thm:min_resol} and \ref{thm:min_resol2}.
\end{rem}

To prove Proposition \ref{prop:stab}, it is convenient to use the concept of axiomatic Fitting invariants introduced by the first author \cite{Kata_05}.
More concretely, inspired by \cite[\S 3.2]{GK_20}, we use the notion of $\cP$-trivial Fitting invariant defined as follows.

\begin{defn}
A $\cP$-trivial Fitting invariant is a map $\cF: \cC \to \Omega$, where $\Omega$ is a commutative monoid, satisfying the following properties:
\begin{itemize}
\item[(i)]
If $P \in \cP$, we have $\cF(P)$ is the identity element of $\Omega$.
\item[(ii)]
For a short exact sequence $0 \to M' \to M \to P \to 0$ in $\cC$ with $P \in \cP$, we have $\cF(M') = \cF(M)$.
\item[(iii)]
For a short exact sequence $0 \to P \to M \to M' \to 0$ in $\cC$ with $P \in \cP$, we have $\cF(M') = \cF(M)$.
\end{itemize}
It is an important fact \cite[Proposition 3.17]{Kata_05} that conditions (ii) and (iii) are equivalent to each other (assuming (i)).
Note that in this setting we do not have to assume $\Omega$ is a commutative monoid, and instead a pointed set structure suffices.
\end{defn}

A fundamental example of a Fitting invariant is of course given by the Fitting ideal; more precisely, the Fitting ideal modulo principal ideals satisfies the axioms of $\cP$-trivial Fitting invariants.

The following proposition introduces another kind of Fitting invariants.

\begin{prop}\label{prop:stab2}
For $n \geq 2$, define $\cF_n: \cC \to \N$ by $\cF_n(M) = r_n(M)$.
We also define $\cF_{0, 1}: \cC \to \Z$ by $\cF_{0, 1}(M) = r_1(M) - r_0(M)$.
Then these maps $\cF_n$ and $\cF_{0, 1}$ are all $\cP$-trivial Fitting invariants.
\end{prop}

\begin{proof}
We check the conditions (i) and (ii).
Firstly, (i) is a restatement of Lemma \ref{lem:Fitt_inv1}.
Secondly, (ii) follows from the associated long exact sequence, taking Lemma \ref{lem:Fitt_inv1} into account again.
Indeed, the long exact sequence collapses into isomorphisms for degree $\geq 2$ and a $6$-term exact sequence for degrees $0, 1$.
\end{proof}

Note that (iii) cannot be shown in a similar manner.
This is because the lower degree part of the associated long exact sequence becomes an $8$-term exact sequence.
It is important that (iii) follows from (i) and (ii).

\begin{proof}[Proof of Proposition \ref{prop:stab}]
By Proposition \ref{prop:stab2}, it is enough to show $\cF(N) = \cF(N')$ (resp.~$\cF(M) = \cF(M')$) if $M \simeq M'$ (resp.~$N \simeq N'$) for any $\cP$-trivial Fitting invariant $\cF$.
For this, we apply the theory of shifts $\cF^{\langle d \rangle}(-)$, $\cF^{\langle -d \rangle}(-)$ of Fitting invariants of the first author \cite[Theorem 3.19]{Kata_05}.
By the exact sequence involving $M$, $N$, and $P_i$, the definition of the shifts implies
\[
\cF^{\langle d \rangle}(M) = \cF(N),
\quad
\cF^{\langle -d \rangle}(N) = \cF(M),
\]
and similarly for $M'$, $N'$.
Then what we have to show is just a reformulation of the well-definedness of the shifts, which is already established by the first author in \cite[Theorem 3.19]{Kata_05}.
\end{proof}

\subsection{Proof of Theorem \ref{thm:min_resol2}}\label{ss:resol2}

\begin{proof}[Proof of Theorem \ref{thm:min_resol2}]
(1)
By taking the dual of the Tate sequence, 
we obtain an exact sequence
\begin{equation}\label{eq:dual_Tate}
0 \to \Z_p \to Q^{*} \to P^{*} \to X^{*} \to 0,
\end{equation}
where we used $\Z_p^{*} \simeq \Z_p$.

As in \S \ref{ss:pf_spec}, we use the specialization method.
Let us take any element $f \in \cF_0$ that is coprime to $\ch_{\Lambda}(P)$ and $\ch_{\Lambda}(Q)$.
Put $m = \ord_p(f(0)) \geq 1$.
Then \eqref{eq:dual_Tate} yields an exact sequence
\[
0 \to \Z_p/p^m \Z_p \to Q^{*}/f \to P^{*}/f \to X^{*}/f \to 0.
\]
Observe that both $P^{*}/f$ and $Q^{*}/f$ are $G$-cohomologically trivial.
So we have
\[
\hat{H}^{-1}(G, X^{*}/f) \simeq H^1(G, \Z_p/p^m\Z_p),
\]
where $\hat{H}^{-1}$ denotes the Tate cohomology group.
By definition, $\hat{H}^{-1}(G, X^*/f)$ is a submodule of $H_0(G, X^*/f)$, so the above isomorphism shows
\[
\gen_{\cO_f}(H_0(G, X^{*}/f)) \geq \gen_{\cO_f}(H^1(G, \Z_p/p^m\Z_p))
\]
as $\cO_f$ is a DVR.
By Nakayama's lemma, the left hand side is equal to $\gen_{\Lambda[G]}(X^{*})$.
Also, as in the proof of Theorem \ref{thm:min_resol}(1) in \S \ref{ss:pf_spec},
\begin{align}
\gen_{\cO_f}(H^1(G, \Z_p/p^m\Z_p))
& = \gen_{\Z_p/p^m\Z_p}(H^1(G, \Z_p/p^m\Z_p))\\
& = \gen_{\Z_p/p^m\Z_p}(H_1(G, \Z_p/p^m\Z_p))\\
& = \dim_{\F_p}(H_1(G, \F_p))\\
& = s_1,
\end{align}
where the second equality follows from Lemma \ref{lem:HN1} and the third from Lemma \ref{lem:H1}.
Thus we obtain (1).

(2)
In case $G$ is trivial, since the projective dimension of $X$ as a $\Lambda$-module is $\leq 1$, we may apply Lemma \ref{lem:Fitt_inv1} to obtain the assertion.

From now on, we assume $G$ is non-trivial.
Since $s_n = r_n^{\Z_p[G]}(\Z_p)$ by Lemma \ref{lem:hom_r}, a minimal resolution of $\Z_p$ as a $\Z_p[G]$-module is of the form
\begin{equation}\label{eq:resol_Zp}
\cdots \to \Z_p[G]^{s_2} \overset{d_2}{\to} \Z_p[G]^{s_1} \overset{d_1}{\to} \Z_p[G]^{s_0} \overset{\varepsilon}{\to} \Z_p \to 0.
\end{equation}
We truncate it to an exact sequence
\[
0 \to \Ker(d_1) \to \Z_p[G]^{s_1} \overset{d_1}{\to} \Z_p[G]^{s_0} \overset{\varepsilon}{\to} \Z_p \to 0.
\]
Since its dual $(-)^{*}$ is also exact and we have $\Z_p^{*} \simeq \Z_p$ and $\Z_p[G]^{*} \simeq \Z_p[G]$, we obtain an exact sequence
\begin{equation}\label{eq:resol_Kerd}
0 \to \Z_p \overset{\varepsilon^{*}}{\to} \Z_p[G]^{s_0} \overset{d_1^{*}}{\to} \Z_p[G]^{s_1} \to \Ker(d_1)^{*} \to 0.
\end{equation}
By comparing \eqref{eq:dual_Tate} and \eqref{eq:resol_Kerd}, Proposition \ref{prop:stab}(2) implies
\[
r_1(X^{*}) - r_0(X^{*}) = r_1(\Ker(d_1)^{*}) - r_0(\Ker(d_1)^{*})
\]
and
\[
r_n(X^{*}) = r_n(\Ker(d_1)^{*})
\]
for $n \geq 2$.

Let us compute $r_n(\Ker(d_1)^{*})$ for any $n \geq 0$.
We combine \eqref{eq:resol_Kerd} with \eqref{eq:resol_Zp} to an exact sequence
\[
\cdots \to \Z_p[G]^{s_2} \overset{d_2}{\to} \Z_p[G]^{s_1} \overset{d_1}{\to} \Z_p[G]^{s_0} \overset{\varepsilon^{*} \circ \varepsilon}{\to}
\Z_p[G]^{s_0} \overset{d_1^{*}}{\to} \Z_p[G]^{s_1} \to \Ker(d_1)^{*} \to 0.
\]
By construction, this is a minimal resolution of $\Ker(d_1)^{*}$ as a $\Z_p[G]$-module.
For this we need the hypothesis that $G$ is non-trivial; the map $\varepsilon^{*} \circ \varepsilon$ can be identified with the map $\Z_p[G] \to \Z_p[G]$ that sends $1$ to the norm element $N_G = \sum_{g \in G} g$, and $N_G$ is in the Jacobson radical of $\Z_p[G]$ if and only if $G$ is non-trivial.
Therefore, we obtain
\[
r_n^{\Z_p[G]}(\Ker(d_1)^{*}) = 
\begin{cases}
	s_1 & (n = 0)\\
	s_0 & (n = 1, 2)\\
	s_{n-2} & (n \geq 3).
\end{cases}
\]
By applying Lemma \ref{lem:resol_bc}, we obtain
\[
r_n^{\Lambda[G]}(\Ker(d_1)^{*}) = 
\begin{cases}
	s_1 & (n = 0)\\
	s_1 + s_0 & (n = 1)\\
	s_0 + s_0 & (n = 2)\\
	s_{n-2} + s_{n-3} & (n \geq 3).
\end{cases}
\]
This completes the proof.
\end{proof}

\section{Proof of the rest of Theorems \ref{thm:bound2}(1) and \ref{thm:gen_dual}(1)}\label{sec:pf}

Now we return to the arithmetic situation described in \S \ref{ss:set}.

\subsection{Proof of Theorem \ref{thm:gen_dual}(1)}\label{ss:dual}

First we recall the following well-known fact.

\begin{lem}\label{lem:Gab}
We have an isomorphism
\[
(X_{k_{\infty}, S})_{\Gal(k_{\infty}/k)}
\simeq \Gal(M_{k, S}/k_{\infty}).
\]
In particular, by Nakayama's lemma we have
\[
\gen_{\Lambda}(X_{k_{\infty}, S}) = \rank_p\Gal(M_{k, S}/k_{\infty}).
\]
\end{lem}

To prove Theorem \ref{thm:gen_dual}(1), we also need the following general lemma.

\begin{lem}\label{lem:dual_pd1}
As in \S \ref{sec:abst_Tate}, consider $\Lambda = \Z_p[[T]]$ and a finite $p$-group $G$.
For a finitely generated torsion $\Lambda[G]$-module $M$ whose projective dimension is $\leq 1$, we have
\[
\gen_{\Lambda[G]}(M^*)
= \gen_{\Lambda[G]}(M).
\]
\end{lem}

\begin{proof}
As in Lemma \ref{lem:Fitt_inv1}, the minimal resolution of $M$ is of the form $0 \to \Lambda[G]^a \to \Lambda[G]^a \to M \to 0$.
By taking $\Ext$-functor, this induces an exact sequence $0 \to \Lambda[G]^a \to \Lambda[G]^a \to M^* \to 0$, which is again a minimal resolution.
Thus we obtain the lemma.
\end{proof}

\begin{proof}[Proof of Theorem \ref{thm:gen_dual}(1)]
Applying the dual of the Tate sequence introduced in (\ref{eq:dual_Tate}) to our setting, we obtain an exact sequence
\[
0 \to \Z_p^{*} \to Q^{*} \to P^{*} \to X_{K_{\infty}, S}^{*} \to 0.
\]
By the compatibility of the Tate sequences, we obtain an isomorphism
\[
(X_{K_{\infty}, S}^{*})_G 
\simeq X_{k_{\infty}, S}^{*}.
\]
Therefore, we have
\[
\gen_{\cR}(X_{K_{\infty}, S}^{*})
= \gen_{\Lambda}(X_{k_{\infty}, S}^{*}).
\]
By Lemma \ref{lem:dual_pd1} (with $G$ trivial), this is then equal to $\gen_{\Lambda}(X_{k_{\infty}, S})$.
By Lemma \ref{lem:Gab}, this completes the proof.
\end{proof}

\subsection{Proof of Theorem \ref{thm:bound2}(1)}
Note that applying Theorem \ref{thm:min_resol}(1) to the Tate sequence introduced in Theorem \ref{thm:Tate}, we already obtained the inequality $s_2 \leq \gen_{\cR}(X_{K_{\infty}, S})$.
Therefore, it remains only to prove $t \leq \gen_{\cR}(X_{K_{\infty},S}) \leq s_2 + t$.

\begin{lem} \label{lem:coh1}
We have an exact sequence 
$$
0 \to H_{2}(G, \Z) \to (X_{K_{\infty},S})_{G} \to X_{k_{\infty},S} \to G^{\ab} \to 0
$$
of $\Lambda$-modules. 
\end{lem}

\begin{proof}
By the Hochschild--Serre spectral sequence we have an exact sequence 
$$
0 \to H^1(G, \Q_p/\Z_p) \to H^1({\mathcal O}_{K_{\infty},S},\Q_p/\Z_p) \to 
H^1({\mathcal O}_{k_{\infty},S},\Q_p/\Z_p)^G \to H^2(G, \Q_p/\Z_p) \to 0,
$$
where we used the weak Leopoldt conjecture $H^2({\mathcal O}_{K_{\infty},S},\Q_p/\Z_p)=0$. 
For $i = 1, 2$, the Pontryagin dual of $H^i(G, \Q_p/\Z_p)$ is the projective limit of 
$H_{i}(G, \Z/p^m)$ by Lemma \ref{lem:HN1}. 
But since $G$ is a finite $p$-group, 
it is finite and isomorphic to $H_{i}(G, \Z)$.
Therefore, taking the Pontryagin dual of the above exact sequence and using Lemma \ref{lem:H3}, we get the conclusion.
\end{proof}

Note that $G^{\ab} = \Gal(M_{k_{\infty}, S} \cap K_{\infty}/k_{\infty})$.
Then putting $Y=\Gal(M_{k_{\infty},S}/M_{k_{\infty}, S} \cap K_{\infty})$, by Lemma \ref{lem:coh1} we obtain two exact sequences
\[
0 \to H_{2}(G, \Z) \to (X_{K_{\infty},S})_{G} \to Y \to 0,
\quad
0 \to Y \to X_{k_{\infty},S} \to G^{\ab} \to 0.
\]
The first sequence implies
\begin{equation}\label{eq:Y1}
\gen_{\Lambda}(Y) 
\leq \gen_{\cR}(X_{K_{\infty},S})
\leq \gen_{\Z_p}(H_{2}(G, \Z)) + \gen_{\Lambda}(Y).
\end{equation}
On the other hand, by taking the $\Gamma=\Gal(k_{\infty}/k)$-coinvariant of the second sequence, we obtain an exact sequence
\[
G^{\ab} \to Y_{\Gamma}  \to (X_{k_{\infty},S})_{\Gamma} \to G^{\ab} \to 0.
\]
By Lemma \ref{lem:Gab}, this is reformulated as
\[
G^{\ab} \to Y_{\Gamma}  \to \Gal(M_{k, S}/M_{k_{\infty}, S} \cap K_{\infty}) \to 0.
\]
This sequence, together with $M_{k_{\infty},S} \cap K_{\infty}=M_{k,S} \cap K_{\infty}$ and the definition of $t$, implies that
\begin{equation}\label{eq:Y2}
t \leq \gen_{\Lambda}(Y) \leq \gen_{\Z_p}(G^{\ab})+ t.
\end{equation}
By combining \eqref{eq:Y1} and \eqref{eq:Y2}, we obtain
\[
t 
\leq \gen_{\cR}(X_{K_{\infty},S})
\leq \gen_{\Z_p}(H_{2}(G, \Z)) + \gen_{\Z_p}(G^{\ab})+ t
= s_2 + t,
\]
where we used Lemma \ref{lem:H1} to get the final equality.
Note that if $G$ is abelian, we have $H_{2}(G, \Z) \simeq \bigwedge^2 G$ and 
$\gen_{\Z_p}(\bigwedge^2 G)=s(s-1)/2$ (with $s = \rank_p G$), which imply explicitly
$$s+\gen_{\Z_p}(H_{2}(G, \Z))=s + s(s-1)/2 = s(s+1)/2=s_2.$$
This completes the proof of Theorem \ref{thm:bound2}(1).

\section{Numerical Examples}\label{sec:examples}

In this section, we numerically check the inequality concerning $\gen_{\cR}(X_{K_{\infty}, S})$ in Theorem \ref{thm:bound1} by using the computer package PARI/GP.
We consider $k = \Q$ and its finite abelian $p$-extension $K$ that is totally real.
Let $S$ be a finite set of places of $\Q$ containing $p$ and the archimedean place $\infty$, such that $K/\Q$ is unramified outside $S$.

The basic method is as follows.
First, by Nakayama's lemma and Lemma \ref{lem:Gab}, we have 
\[
\gen_{\Z_p[[\Gal(K_{\infty}/\Q)]]}(X_{K_{\infty}, S})
= \gen_{\Z_p}((X_{K_{\infty}, S})_{\Gal(K_{\infty}/\Q)})
= \gen_{\Z_p}(\Gal(M_{K, S}/K_{\infty})_{\Gal(K/\Q)}).
\]
We observe that $M_{K, S}$ is the union of the maximal $p$-extensions of $K$ in the ray class fields of modulus $p^m \prod_{v \in S \setminus \{p\}} v$ for all $m \geq 0$.
Note here that, since $K$ is abelian over $\Q$, the Leopoldt conjecture is shown to be true for $K$ by work of Brumer (see \cite[Theorem 10.3.16]{NSW08}).
Therefore, $\Gal(M_{K, S}/K_{\infty})$ is finite, and we can compute it by computing the ray class groups for finitely many $m$.
In this way we can determine the quantity $\gen_{\Z_p}(\Gal(M_{K, S}/K_{\infty})_{\Gal(K/\Q)})$.

\subsection{The case $p = 3$}

Let us take $p = 3$, though the discussion is basically valid for any odd prime $p$.
We write $S \setminus \{p, \infty\} = \{\ell_1, \dots, \ell_s\}$.
By the theorem of Kronecker--Weber, the Galois group of the maximal abelian extension of $\Q$ unramified outside $S$ is isomorphic to $\Z_p^{\times} \times \prod_{i = 1}^s \Z_{\ell_i}^{\times}$, so we have
\[
\Gal(M_{\Q, S}/\Q) \simeq	\Z_p \times \prod_{i = 1}^s \Z_p/(\ell_i-1)\Z_p.
\]
As is well-known, we may assume $\ell_i \equiv 1 (\bmod \ p)$ for $1 \leq i \leq s$ without loss of generality.
For such an $S$, let us take $K$ as the unique intermediate field of $\Q(\mu_{\ell_1}, \dots, \mu_{\ell_s})/\Q$ such that
\[
\Gal(K/\Q) \simeq (\Z/p\Z)^s.
\]
In this case, the above information on $M_{\Q, S}$ implies that
\[
\Gal(M_{\Q, S}/K_{\infty}) \simeq \prod_{i = 1}^s p\Z_p/(\ell_i-1)\Z_p,
\]
so the integer $t$ in Theorem \ref{thm:bound1} is determined as $t = \# \{1 \leq i \leq s \mid \ell_i \equiv 1 (\bmod \  p^2)\}$.
To ease the notation, we write $\gen(X) = \gen_{\Z_p[[\Gal(K_{\infty}/\Q)]]}(X_{K_{\infty}, S})$. 
Then Theorem \ref{thm:bound1} asserts 
\[
\max \left\{ \frac{s(s+1)}{2}, t \right\} \leq \gen(X) \leq \frac{s(s+1)}{2} + t.
\]
Let $P$ be the set of all prime numbers $\ell$ such that $\ell \equiv 1 (\bmod \  p)$.
We divide $P$ into two subsets $P_1$ and $P_2$ defined by 
\[
P_1=\{\ell \in P \mid \ell \not \equiv 1 (\bmod \ p^2)\}=
\{7, 13, 31, 43, 61, 67, 79, 97, \dots.\}
\]
and 
\[
P_{2}=\{\ell \in P  \mid \ell \equiv 1 (\bmod \ p^2)\}=\{19, 37, 73, 109, 127, 163, 181, 199, \dots.\}.
\]
Then we have $t = \# \{1 \leq i \leq s \mid \ell_i \in P_2\}$.

First we consider the case $t = 0$, that is, $\ell_{1}, \dots, \ell_{s}$ are $s$ distinct primes in $P_1$.  
Then Theorem \ref{thm:bound1} asserts $\gen(X) = s(s+1)/2$ as in Remark \ref{rem:t=0}(1).
By using PARI/GP, we numerically checked $\gen(X) = 1$ if $s = 1$ and $\ell_1 \leq 100$, and $\gen(X) = 3$ if $s = 2$ and $\ell_1, \ell_2 \leq 100$, as Theorem \ref{thm:bound1} says.  


Suppose that $t = 1$ and $s = 1$, that is, $S=\{3, \ell_1, \infty\}$ with $\ell_1 \in P_2$. 
Then Theorem \ref{thm:bound1} asserts 
\[
1 \leq \gen(X) \leq 2.
\]
By numerical computation, we find only $\gen(X) = 2$ in the range $\ell_1 \leq 200$, 
and did not encounter $\gen(X) = 1$.

Next consider the case $t = 1$ and $s = 2$, that is, $S=\{3, \ell_1, \ell_2, \infty\}$ with $\ell_1 \in P_1$ and $\ell_2 \in P_2$. 
Then we have 
\[
3 \leq \gen(X) \leq 4
\]
by Theorem \ref{thm:bound1}.
In the range $\ell_1 < 100$ and $\ell_2 < 200$, we find $\gen(X)=3$ except for
\begin{align}
(\ell_1, \ell_2) = 
& (7, 127), (7, 181), (13, 73), (13, 109), (13, 181), (31, 109), (31, 163), (43, 127),\\
&  (43, 199), (61, 37), (61, 163), (67, 109), (79, 199), (97, 19), (97, 109), (97, 127),
\end{align}
for which $\gen(X) = 4$.
Thus the above inequality on $\gen(X)$ is sharp in this case.

Finally we consider the case $t = 2$ and $s = 2$, that is, $S=\{3, \ell_1, \ell_2, \infty\}$ with $\ell_1$, $\ell_2 \in P_2$. 
Then Theorem \ref{thm:bound1} says 
\[
3 \leq \gen(X) \leq 5.
\]
By numerical computation, in the range $\ell_1 < \ell_2 \leq 200$, we find $\gen(X) = 4$ except for $(\ell_1, \ell_2) = (109, 199)$, for which $\gen(X) = 5$.
We did not encounter $\gen(X) = 3$ in this situation.

Due to the limitation of the machine power we could not handle $s \geq 3$ when $p = 3$.

\subsection{The case $p = 2$}

Let us discuss the case $p = 2$.
Suppose that $S=\{2, \ell_1, \dots, \ell_s, \infty\}$ where $\ell_1, \dots, \ell_s$ are $s$ distinct odd prime numbers.
We have
\[
\Gal(M_{\Q, S}/\Q) \simeq	\Z_2 \times \Z/2\Z \times \prod_{i = 1}^s \Z_2/(\ell_i-1) \Z_2.
\]
Let us set $K$ as $K = \Q(\sqrt{\ell_1}, \dots, \sqrt{\ell_s})$, so we have $\Gal(K/\Q) \simeq (\Z/2\Z)^s$ and
\[
\Gal(M_{\Q, S}/K_{\infty}) \simeq \Z/2\Z \times \prod_{i = 1}^s 2\Z_2/(\ell_i-1)\Z_2.
\]
As in the previous subsection, we define $P$ to be the set of all odd prime numbers, $P_2$ to be the subset of $P$ consisting of $\ell$ such that $\ell \equiv 1 (\bmod \ 4)$, and $P_1=P \setminus P_2$. 
The integer $t$ in Theorem \ref{thm:bound1} is then $t = 1 + \# \{1 \leq i \leq s \mid \ell_i \in P_2\}$.

We first consider the case $s=1$, so $K=\Q(\sqrt{\ell_1})$.
If $\ell_1$ is in $P_1$, then $t=1$ and 
Theorem \ref{thm:bound1} says $1 \leq \gen(X) \leq 2$.
For $\ell_1 \in P_1$ less than $100$, we always have $\gen(X) = 2$.
If $\ell_1$ is in $P_2$, then $t=2$ and 
Theorem \ref{thm:bound1} says $2 \leq \gen(X) \leq 3$.
For $\ell_1 \in P_2$ less than $100$, we have $\gen(X) = 2$ except for $\ell_1=73$, $89$, $97$, for which we have $\gen(X) = 3$.
So in this case the inequality is sharp.

In the following we focus on the case $\ell_1, \dots, \ell_s$ are all in 
\[
P_2=\{5, 13, 17, 29, 37, 41, 53, 61, 73, 89, 97, \dots.\}.
\] 
Then $t=s+1$, and the inequality in Theorem \ref{thm:bound1} becomes 
\[
\frac{s(s+1)}{2} \leq \gen(X) \leq \frac{s(s+1)}{2} + s + 1
\] 
for $s \geq 2$.

For $s=2$, we have 
\[
3 \leq \gen(X) \leq 6
\]
by Theorem \ref{thm:bound1}.
By numerical computation, we find $\gen(X) = 4, 5, 6$ in the range $\ell_1 < \ell_2 \leq 100$.
Concretely, we have $\gen(X) = 6$ if 
\[
(\ell_1, \ell_2) = (17, 89), (41, 73), (73, 89), (73, 97), (89, 97),
\]
$\gen(X) = 5$ if 
\begin{align}
(\ell_1, \ell_2) = 
& (5, 41), (5, 89), (13, 17), (17, 41), (17, 53), (17, 73), (17, 97), (37, 41), \\
& (37, 73), (41, 61), (41, 89), (41, 97), (53, 89), (53, 97), (61, 73), (61, 97),
\end{align}
and $\gen(X) = 4$ otherwise.

For $s=3$, Theorem \ref{thm:bound1} says  
\[
6 \leq \gen(X) \leq 10.
\]
For $\ell_1 = 5$ and $5 < \ell_2 < \ell_3 \leq 100$, 
we have $\gen(X)=7$ except for
\[
(\ell_2,\ell_3) = 
(17, 89), (37,41), (41,61), (41, 73), (41, 89), (53, 89), (73, 89), (89, 97),
\]
for which we have $\gen(X)=8$.
Also, we find examples for $\gen(X) = 9, 10$ by taking respectively $(\ell_1, \ell_2, \ell_3) = (17, 73, 89), (73, 89, 97)$.


\subsection{A variant}
As a final remark, let us briefly discuss a variant that matters only when $p = 2$.
So far we always assumed that $S$ contains all archimedean places, so we studied the narrow class groups.
Theoretically this assumption is necessary to use the Tate sequence in Theorem \ref{thm:Tate}.
However, the numerical computation in this section is possible (and simpler) even if we remove the archimedean places from $S$.

Suppose that $p=2$, $k=\Q$, $K=\Q(\sqrt{\ell_1}, \dots, \sqrt{\ell_s})$ as in the previous subsection, and $S'=S \setminus \{\infty\}$.
We consider $\gen_{\Z_p[[\Gal(K_{\infty}/\Q)]]}(X_{K_{\infty}, S'})$, which we abbreviate as $\gen(X')$.
We have a natural surjective homomorphism from $X_{K_{\infty}, S}$ to $X_{K_{\infty}, S'}$ whose kernel is a cyclic module (since $k = \Q$).
Therefore, we have
\[
\gen(X') \leq \gen(X) \leq \gen(X') + 1.
\]
Still assuming $\ell_i \equiv 1 (\bmod \ 4)$ for any $1 \leq i \leq s$, we find the following numerical examples.
\begin{itemize}
\item
When $s = 1$, we find examples for $\gen(X') = 1, 2$.
\item
When $s = 2$, we find examples for $\gen(X') =3, 4, 5$.
\item
When $s = 3$, we find examples for $\gen(X') = 6, 7, 8, 9$.
\end{itemize}
These results suggest that $s(s+1)/2 \leq \gen(X') \leq s(s+1)/2 + s$, but this does not follow directly from Theorem \ref{thm:bound1}.
The above computations suggest that Theorem \ref{thm:bound1} holds true without assumption that $S$ contains all archimedean places.

{
\bibliographystyle{abbrv}
\bibliography{biblio}
}

\end{document}